\documentclass[12pt,a4paper,reqno]{amsart}

\usepackage{amssymb}
\usepackage{enumerate}
\usepackage[all]{xy}
\newtheorem{theorem}{Theorem}
\newtheorem{proposition}{Proposition}[section]
\newtheorem{lemma}[proposition]{Lemma}
\newtheorem{corollary}[proposition]{Corollary}
\newtheorem{question}[proposition]{Question}

\theoremstyle{definition}
\newtheorem{definition}[proposition]{Definition}
\newtheorem{remark}[proposition]{Remark}
\newtheorem{example}[proposition]{Example}

\newcommand{\pt}{\mathbf {x}}
\newcommand{\Proj}{\mathrm{Proj }}

\newcommand{\cF}{\mathcal F}

\newcommand{\fm}{\mathfrak{m}}

\newcommand{\fb}{\mathfrak{b}}
\newcommand{\fc}{\mathfrak{c}}

\newcommand{\PP}{\mathbb{P}}
\newcommand{\p}{\mathbb{P}}

\newcommand{\Z}{\mathbb{Z}}

\newcommand{\Aut}{\mathrm{Aut}}

\newcommand{\SL}{\mathrm{SL}}

\newcommand{\A}{\mathbb{A}}

\newcommand{\C}{\mathbb{C}}

\newcommand{\N}{\mathbb{N}}
\newcommand{\G}{\mathbb{G}}

\newcommand{\cO}{\mathcal O}
\newcommand{\Spec}{\mathrm{Spec}}
\newcommand{\bideg}{\mathrm{bideg}}

\newcommand{\gr}{\mathrm{gr}}

\newcommand{\Bl}{\mathrm{Bl}}

\newcommand{\longto}{\to}
\renewcommand{\epsilon}{\varepsilon}
\renewcommand{\phi}{\varphi}

\newcommand{\quot}{/\!\!/}
\title[Affine extensions of principal additive bundles over a punctured surface]{Affine extensions of principal additive bundles over a punctured surface}

\author{Isac Hed\'en}
 
\address{Isac Hed\'en\\
 Mathematisches Institut\\
Universit\"at Basel\\
 Rheinsprung 21\\
  4051 Basel, Switzerland}
\email{Isac.Heden@unibas.ch}

\thanks{This work was done as part of PhD studies at the Department of Mathematics, Uppsala University; the support from the Swedish graduate school in Mathematics and Computing (FMB) is gratefully acknowledged. Many thanks also go to Karl-Heinz Fieseler for his supervision.}
\date{\today}

\begin{document}
\begin{abstract}{The aim of this article is to make a first step towards the classification of complex normal affine $\G_a$-threefolds $X$. We consider the case where the  restriction of the quotient morphism $\pi\colon X\to S$ to $\pi^{-1}(S_*)$, where $S_*$ denotes the  complement of some regular closed point in $S$, is a principal $\G_a$-bundle. The variety $\SL_2$ will be of special interest and a source of many examples. It has a natural right $\G_a$-action such that the quotient morphism $\SL_2\to\A^2$ restricts to a principal $\G_a$-bundle over the punctured plane $\A^2_*$.}\end{abstract}

\maketitle

\vspace{-5ex}
\section{Introduction}

Given a complex normal affine variety $X$ with an algebraic $\G_a$-action, the ring $\cO(X)^{\G_a}$ of invariants is finitely generated if $\dim X\leq 3$ \cite[p. 45]{Nag59}, so we can always define a quotient variety $X\quot\G_a:=\Spec(\cO(X)^{\G_a})$ in this case. This quotient variety is of dimension $\dim X-1$ unless the $\G_a$-action is the trivial one.
Let $\pi\colon X\to X\quot\G_a$ be the quotient morphism, and denote by $(X\quot\G_a)_*\subset X\quot\G_a$ the union of all open subsets $U_i\subset X\quot\G_a$ over which there is a $\G_a$-invariant trivialization $\pi^{-1}(U_i)\stackrel\sim\longto U_i\times\G_a$. Then $(X\quot\G_a)_*$ is the maximal open subset $V\subset X\quot\G_a$ such that $\pi|_{\pi^{-1}(V)}\colon \pi^{-1}(V)\to V$ is a principal $\G_a$-bundle; this set is always nonempty.

If $X$ is a surface, the quotient morphism is surjective and $(X\quot\G_a)_*$ is affine. In particular $\pi^{-1}((X\quot\G_a)_*)\subset X$ is equivariantly isomorphic to $(X\quot\G_a)_*\times\G_a$, where $\G_a$ acts by translation on the second factor. It is shown in \cite{Fie94} that complex normal affine $\G_a$-surfaces are classified by the quotient $X\quot\G_a$ and neighbourhoods of the fibers of the points in $X\quot\G_a\setminus (X\quot\G_a)_*$.\par
If $X$ is a normal threefold, the quotient $S:=X\quot\G_a$ is a normal affine surface, and we will study the threefolds for which $(X\quot\G_a)_*=S\setminus\{\pt\}$ for some closed regular point $\pt\in S$. In order to do this in a systematic way, we introduce the following notion.

\begin{definition}
\label{def:affext}
Let $S_*\subset S$ be the open subvariety of an affine normal surface $S$ which is obtained by removing a closed regular point $\pt$. An affine extension of a principal $\G_a$-bundle $\pi\colon P\to S_*$ is a normal affine $\G_a$-variety $\hat P=\Spec(B)$ together with a morphism $\hat\pi\colon\hat P\to S$ and a $\G_a$-equivariant dominant open embedding $\iota\colon P\hookrightarrow\hat P$ with $\iota (P)= \hat \pi^{-1}(S_*)$, which makes the  following diagram commute \\
$$
\begin{minipage}{0.32\textwidth}
\[\xymatrix{
P\ar@{^{(}->}[r]^\iota\ar[d]_{\pi} & \hat P\ar[d]^{\hat\pi}\\
S_*\ar@{^{(}->}[r] & S}\]
\end{minipage}\ .
$$

We will use the notation $E=\hat\pi^{-1}(\pt)$, $A=\cO(P)$,  $B=\cO(\hat P)\stackrel{\iota^*}\hookrightarrow A$, and $\fm_\pt\subset\cO(S)$ for the exceptional fiber, the regular functions on $P$, the subalgebra of regular functions on $P$ that extend to $\hat P$, and the maximal ideal of $\pt$, respectively.
\end{definition}
\begin{remark} The diagram from Definition~\ref{def:affext} on the algebraic side looks as follows 
\[\xymatrix{
\cO(P)^{\G_a}&\cO(\hat P)^{\G_a} \ar@{_{(}->}[l]_{\iota^*} \\
\cO(S_*)\ar[u]^{\simeq}  & \cO(S)\ar[l]_{\simeq}\ar@{_{(}->}[u]_{{\hat\pi}^*}}\]
if we take restrictions to the invariant algebras. It follows in particular that $\hat P\quot\G_a = S$ and that $\hat \pi\colon\hat P\to S$ is the quotient morphism.
\end{remark}

To start with, we devote section~\ref{sec:trivial} to extensions of the trivial principal $\G_a$-bundle $S_*\times\G_a\to S_*$.
\begin{theorem}
\label{thm:2cases}

For an affine extension $\hat P$ of the trivial bundle $S_* \times \G_a$, the morphism $\iota$ extends to a morphism $j\colon S \times \G_a \to \hat P$ which is either an open embedding
or contracts $\{\pt\}\times\G_a$ to a singular point $p_0\in\hat P$. In the first case either $j$ is an isomorphism or $E=j(\{\pt\}\times \G_a) \cup E_{(2)}$ is a disjoint union with a purely two dimensional set $E_{(2)}$; in the second case the exceptional fiber $E$ is purely two dimensional.
\end{theorem}
We talk accordingly of extensions of the ''first kind'' ($j$ is an open embedding) and extensions of the ''second kind'' ($j$ contracts $\{\pt\}\times\G_a$). An obvious extension of $S_*\times\G_a$ of the first kind is of course $S\times\G_a$, but there are many others! A series of examples of smooth extensions of $\A^2_*\times\G_a$ is presented before we move on to extensions of the second kind. Note that $S_*\times\G_a$ has a natural $\G_m$-action, where $\G_m$ acts trivially on $S_*$ and as $\Aut(\G_a)$ on the fibers. We call this the vertical $\G_m$-action. Examples of affine extensions of $S_*\times\G_a$ such that the vertical $\G_m$-action extends to $\hat P$ are also given.

\begin{theorem}\label{thm:triv2} Let $\hat P\not\simeq S\times\G_a$ be an affine extension of $S_*\times\G_a$ which admits an extension of the vertical $\G_m$-action. Then $\hat P$ is of the second kind and $E= \hat P^{\G_a}$ is the set of $\G_a$-fixed points. 
On the other hand
$p_0:=j(\{\pt\}\times\G_a)$ is the unique $\G_m$-fixed point in the exceptional fiber $E$. Furthermore all the irreducible components $E_i \hookrightarrow E$ of the exceptional fiber contain $p_0$ and, for each $i$, $E_i \setminus \{p_0\}$ is a $\G_m$-fibration over a rational curve. 
\end{theorem}
A $\G_m$-fibration in this context is an affine morphism $q\colon X\to Y$, where $X$ is a $\G_m$-variety, such that the fibers are the $\G_m$-orbits and such that $q^{-1}(V)\quot\G_m\simeq V$ for each affine open subset $V\subset Y$.\par
The ''most basic'' nontrivial principal $\G_a$-bundle over $\A^2_*$ is $\SL_2\to\A^2_*,\quad A=(a_{ij})\mapsto(a_{11},a_{21})$. Recall that $\G_a$ embeds in $\SL_2$ as the upper-triangular unipotent matrices; the action of $\G_a$ on $\SL_2$ is given by right multiplication. Now, if $P \to S_*$ is any nontrivial principal $\G_a$-bundle, we show in section~\ref{sec:SL2} that it is possible to find a punctured neighbourhood $U_*= U \setminus \{\pt \} \subset S$ of $\pt$ together with a morphism  $\varphi =(g,h)\colon U_* \to \A^2_*$, such that $P|_{U_*}  =\varphi^* (\SL_2)$. We also show that an affine extension of $\SL_2$ always induces an extension $\tilde P \to U$ of
$\pi^{-1}(U_*) \to U_*$, which patches together with $P \to S_*$ to an affine extension $\hat P \to S$. This is used as motivation for restricting our attention to extensions of $\SL_2$ for the rest of the article.\par

The locally nilpotent derivation $D\colon B\to B$ which corresponds to the $\G_a$-action on the affine variety $\Spec(B)$ can be used to define a graded algebra $\gr_D(B)$, corresponding to the filtration given by $B_{\leq\nu}:=\ker D^{\nu+1}\subset B$. This is done in section~\ref{sec:gradedalgebra}, where we associate to an affine extension $\hat P=\Spec(B)$ its graded algebra $\gr_D(B)$.
Proposition~\ref{prop:grad} says that these graded algebras are given by a certain sequence of ideals $\{\fm_\nu(B)\}_{n\in\N}$ in $\cO(S)$, and it turns out that they are all of the kind that appears as the algebra of an extension of the trivial bundle $S_*\times\G_a$ with extending vertical $\G_m$-action. We also show that $\gr_D(B)$ uniquely determines $B$ if it is generated in degree 1. Finally we formulate Theorem~\ref{Th: MT}, which could be said to be the main theorem of this work. It gives two families of graded subalgebras of $\gr_D(\cO(\SL_2))$ that actually occur as graded algebras of affine extensions of $\SL_2\to\A^2_*$. The construction of these two families is the topic of sections~\ref{sec:family1} and~\ref{sec:family2}.
The family $\hat P_n$ in section~\ref{sec:family1} is indexed by a positive integer $n$, while the family $\hat P(p,q)$ in section~\ref{sec:family2} is indexed by two positive integers $p,q$ which are relatively prime (the first construction also works for $n=0$, but $\hat P_0\simeq\hat P(1,1)$ is listed in the second family instead). Both families are constructed by realizing $\SL_2$ as a fiber bundle over some base and then enlarging the fiber -- in the first case we also need to take the affinization of the obtained variety in order to get an affine extension; this corresponds to contracting a rational curve to a point.
In section~\ref{sec:family1}, the base is $\p^1$ and the fiber a Borel subgroup of $\SL_2$, whereas in section~\ref{sec:family2}, the base is a Danielewski surface and the fiber $\G_m$. In the first family of extensions, the associated graded algebra is generated by its elements of degree 1, and thus uniquely determines the extension. In fact, any other $\SL_2$-extension is a $\G_a$-equivariant modification of one of the $\hat P_n$ in the following sense.

\begin{theorem}
\label{equimod} 
For any affine extension $\hat P$ of $P=\SL_2$, there exists a $\G_a$-equivariant birational morphism $\hat P\longto\hat P_n$ for some $n\in \N$.
\end{theorem}

In the two constructed families of affine extensions of $\SL_2$, the exceptional fiber consists of $\G_a$-fixed points only. In the following section~\ref{sec:otherext}, we obtain further extensions with a free action of $\G_a$ on the exceptional fiber, as well as with a $1$-dimensional fixed point set (the case of isolated fixed points for a $\G_a$-action on an affine variety is not possible).

\section{Extensions of the trivial $\G_a$-bundle}\label{sec:trivial}
Let $P:=S_*\times\G_a$ be the trivial $\G_a$-bundle, let $A:=\cO(S_*)[t]$ be its algebra of regular functions, and let $\hat P$ be an affine extension. Using the canonical isomorphism $\cO(S_*\times\G_a)\simeq\cO(S\times\G_a)$, we can define a morphism $j\colon S\times\G_a\to\hat P$ by the condition $j^*=\iota^*\colon\cO(\hat P)\to\cO(S)[t]$. Note that $j\colon S\times\G_a\to\hat P$ extends $\iota\colon S_*\times\G_a\to\hat P$, and that $B:=\cO(\hat P)\subset A$ is a subalgebra (via $\iota^*$).\par
The algebraic way of formulating the hypotheses that the $\G_a$-action on $S_*\times\G_a$ extends to $\hat P$ and that the morphism $\hat P\to S$ is locally trivial over $S_*$, is to say that the algebra $B$ is invariant with respect to the locally nilpotent derivation $D_t:= \frac{\partial}{\partial t}\colon A \to A$ which corresponds to the $\G_a$-action on $P$, that $\cO(S)\subset B$, and finally that $B_f=A_f$ holds for the localizations with respect to any $f\in\cO(S)$ with $f(\pt)=0$.

\begin{proof}[Proof of Theorem~$\ref{thm:2cases}$] The morphism $j\colon S\times\G_a\to\hat P$ is equivariant since $\iota$ is, and it follows that the restriction $j|_{\{\pt\} \times \G_a}$ is either injective or constant with image $p_0$ for some $p_0 \in \hat P$.\par
Suppose first that it is injective. Then $j\colon S\times\G_a\longto\hat P$ is a birational morphism with finite fibers, and since $\hat P$ is normal, it follows by Zariski's Main Theorem that $j \colon S \times \G_a \to \hat P$ is an open embedding. Furthermore, $E_2 = \hat P \setminus j (S \times \G_a)$ is purely two dimensional, being the complement of an affine open set.\par

If $j(\{\pt\} \times \G_a)=\{p_0\}$, the point $p_0\in\hat P$ is a singularity. Otherwise we could take a non-vanishing three form $\omega$ on some neighbourhood $V$ of $p_0$; its pullback $j^*(\omega)$ would be a three form on the smooth threefold $j^{-1}(V)$, with zero set $j^{-1}(V)\cap (\{\pt\}\times\G_a)$, but this is impossible for dimension reasons. For the last statement, we denote by $E_{(i)}$ the union of the $i$-dimensional irreducible components of $E$, so that $E=E_{(1)}\cup E_{(2)}$. Since $\hat P$ is normal and of dimension 3, we have $\cO(\hat P\setminus E_{(2)})\simeq\cO(\hat P\setminus E)=\cO(S\times\G_a)$, and hence we get a factorization $\hat P\setminus E_{(2)}\to S\times\G_a\to\hat P$ of the inclusion $\hat P\setminus E_{(2)}\hookrightarrow \hat P$, where the second of the maps is $j\colon S\times\G_a\to\hat P$. But since $j\colon S\times\G_a$ contracts $\{\pt\}\times\G_a$ to a point, it follows that $E_{(1)}=\emptyset$, and $E=E_{(2)}$.
\end{proof}

It follows in particular from Theorem~\ref{thm:2cases} that smooth extensions of $\A^2_*\times\G_a$ are of the first kind. An obvious example is $\hat P=\A^2\times\G_a$, but there are many others, as shown by the following construction. Denote by $\underline\A^1$ the affine line with two origins, i.e. the prevariety obtained by gluing $X_1=X_2=\A^1$ along $V_1=V_2=\A^1_*$ via the identity morphism $V_1\to V_2$, and consider a line bundle $\varphi\colon L\longto\underline{\A}^1$ with trivializations $L_i\simeq X_i\times\A^1$ over $X_i$ for $i=1,2$
and transition function $L_1 \to L_2, (x,y)\mapsto(x,x^ny)$ for some 
$n\in\N_{>0}$.
Let $Q \to \underline{\A}^1$ be any affine nontrivial principal $\G_a$-bundle (in fact, every nontrivial $\G_a$-bundle over $\underline{\A}^1$ is affine~\cite[Prop.1.4]{Fie94}) and let $\hat P:= \varphi^*(Q)$ be its pullback with respect to $\varphi$; then $\hat P$ is an affine $\G_a$-variety since $Q$ is, and the natural morphism $\hat P \to Q$ is affine. The principal $\G_a$-bundle $\hat P\to L$ has trivializations over the affine subsets $\varphi^{-1}(X_i)=L_i$, for $i=1,2$, and the trivial principal $\G_a$-bundle is embedded into $\hat P$ via the canonical embedding $\A^2_*\times\G_a\hookrightarrow L_2\times\G_a=\A^2\times\G_a$. The quotient morphism is $\hat P \to L \to \A^2$, where the second arrow is the identity on $L_2\subset L$ and $(x,y)\mapsto(x,x^ny)$ on $L_1\subset L$.

\begin{example}\label{smoothext} We determine $\hat P$ explicitly in a special case. Let $Q_i=X_i\times\G_a$, for $i=1,2$, and let $Q$ be the principal $\G_a$-bundle which is obtained by gluing $Q_1$ and $Q_2$ along $V_1\times\G_a$ and $V_2\times\G_a$ via the morphism $V_1\times\G_a\to V_2\times\G_a,\,(x,t)\mapsto(x,t+\tfrac 1x)$, and let $L$ be the line bundle given by the transition function $V_1\times\A^1\to V_2\times\A^1,\,(x,y)\mapsto(x,xy)$. Then $\hat P$ is obtained by gluing  $L_1\times\G_a$ and $L_2\times\G_a$ along $U_1\times\G_a$ and $U_2\times\G_a$ via the morphism $U_1\times\G_a\to U_2\times\G_a,\,(x,y,t)\mapsto(x,xy,t+\tfrac 1x)$ with $U_i:=\A^1_*\times\A^1\subset L_i$.\par We define a morphism $\eta\colon\hat P\to\A^5$ by 
\begin{eqnarray*}
\eta\colon(x,y,t)&\mapsto&\begin{cases}
(x,xy,xt+1,xyt+y,xt^2+t)&\textrm{ if } (x,y,t)\in L_1\times\G_a\\
(x,y,xt,yt, xt^2-t)&\textrm{ if } (x,y,t)\in L_2\times\G_a.
\end{cases}
\end{eqnarray*}
This is in fact a closed immersion whose image is the irreducible smooth subvariety $Z\hookrightarrow\A^5$ that is given by the three equations $ T_1T_4-T_2T_3=T_2T_5+T_4-T_3T_4=T_1T_5-T_3^2+T_3=0$. The inverse morphism $\eta^{-1}\colon Z\to \hat P$ is given by
\begin{eqnarray*}
\eta^{-1}\colon(a,b,c,d,e)&\mapsto&\begin{cases}
(a,\frac dc,\frac ec)\in L_1\times\G_a&\textrm{ if } c\neq 0\\
(a,b,\frac e{c-1})\in L_2\times\G_a&\textrm{ if } c\neq 1.
\end{cases}
\end{eqnarray*}
It follows that in this example we have $B=\C[x,y,xt,yt, xt^2-t]\subset\C[x,y,t]$.
Note that the vertical $\G_m$-action defined on $\A^2_*\times\G_a\subset L_2\times\G_a$ does extend to $L_2\times\G_a$, but not to all of $\hat P$.
\end{example}

We now pass on to the second kind of extensions of the trivial $\G_a$-bundle.
\begin{lemma}\label{Lemm:constonE}
An affine extension $\hat P$ of $S_*\times\G_a$ is of the second kind if and only if $B$ is a subalgebra of $\cO(S) \oplus \bigoplus_{\nu=1}^\infty \fm_\pt t^\nu.$
\end{lemma}
\begin{proof}
Note that the (non finitely generated) algebra $\cO(S) \oplus \bigoplus_{\nu=1}^\infty \fm_\pt t^\nu\subset\cO(S)[t]$ consists exactly of those functions on $S\times\G_a$ that are constant along $\{\pt\}\times\G_a$. Suppose that $j(\{\pt\}\times\G_a)=\{p_0\}$ for some $p_0\in\hat P$. Then $j^*(f)(\{\pt\}\times\G_a)=\{f(p_0)\}$ for all $f\in\cO(\hat P)$, so $B\subset\cO(S) \oplus \bigoplus_{\nu=1}^\infty \fm_\pt t^\nu$. Conversely, if $j^*(f)$ is constant along $\{\pt\}\times\G_a$ for all $f\in\cO(\hat P)$, it follows that $j$ contracts $\{\pt\}\times\G_a$ since $\cO(\hat P)$ separates points in $\hat P$.
\end{proof}

\begin{lemma}\label{lem: gradedsubB}
Let $\hat P \not\simeq S\times\G_a$ be an affine extension of $S_*\times\G_a$ for which the vertical $\G_m$-action extends. Then $B=\cO(S) \oplus \bigoplus_{\nu=1}^\infty \fm_\nu t^\nu$, with 
a decreasing sequence of $\fm_{\pt}$-primary ideals $\fm_\nu \subset\cO(S)$.
\end{lemma}
\begin{proof} We know that $B$ is a graded subalgebra of $\cO(S)[t]$ with respect to the $t$-grading since the $\G_m$-action extends to $\hat P$, hence it is either $\cO(S)[t]$ or has the given form. The sequence $(\fm_\nu)_{\nu\in\N_{>0}}$ is decreasing since, by assumption,  $B$ is invariant with respect to the locally nilpotent derivation $D_t\colon\cO(S)[t]\to\cO(S)[t]$ which corresponds to the $\G_a$-action on $S_*\times\G_a$. Finally we know that $A_f=B_f$ for all $f\in\cO(S)$ with $f(\pt)=0$, and it follows that $\cO(S)_f\oplus\bigoplus_{\nu=1}^\infty\fm_\nu\cO(S)_ft^\nu=
B_f=A_f=\cO(S)_f\oplus\bigoplus_{\nu=1}^\infty\cO(S)_ft^\nu$. In particular $\cO(S)_f=\fm_\nu\cO(S)_f$, so the support of the $\fm_\nu$ is contained in $\pt$.
\end{proof}
\begin{corollary}\label{cor:thm2firstline}
If $\hat P\not\simeq S\times\G_a$ is an affine extension of $S_*\times\G_a$ such that the vertical $\G_m$-action extends to $\hat P$, then $\hat P$ is of the second kind. 
\end{corollary}
\begin{proof}
This is an immediate consequence of Lemma~\ref{Lemm:constonE} and Lemma~\ref{lem: gradedsubB}, since $\fm_\nu\subset\fm_\pt$ for all $\nu\in\N_{>0}$, the ideals $\fm_\nu$ being $\fm_\pt$-primary.
\end{proof}

Suppose that $B=\cO(S)\oplus\bigoplus_{\nu=1}^\infty\fm_\nu t^\nu$ is the algebra of an affine extension $\hat P$ of $S_*\times\G_a$, as in Lemma~\ref{lem: gradedsubB}. Then we can form the projective spectrum $\Proj(B)$; indeed, it can be thought of as the set of nontrivial $\G_m$-orbits in $\hat P=\Spec(B)$.
  If $B$ is generated in degree 1,
 i.e. if $B=\cO(S)\oplus\bigoplus_{\nu=1}^\infty\fb^\nu t^\nu$ for some ideal $\fb\subset\cO(S)$, the natural map $\hat P \setminus \hat P^{\G_m} \to \Proj (B)$ is even a locally trivial $\G_m$-principal bundle and $\Proj(B)$ is just the blowup of $S$ at the ideal $\fb$.

\begin{proof}[Proof of Theorem~$\ref{thm:triv2}$] We have seen already in Corollary~\ref{cor:thm2firstline} that $j\colon S\times\G_a\to\hat P$ contracts $\{\pt\}\times\G_a$. Denote $Z \to \Proj (B)$ a resolution of the singularities. Then $Z \to S$ is a composition of blowups at regular points, hence the zero fiber has irreducible components isomorphic to $\PP^1$. Thus the irreducible components of the zero fiber of  $\Proj (B) \to S$ are dominated by $\PP^1$, hence are rational curves. 
It remains to show that $\G_a$ acts trivially on $E$.  The standard action of the affine group $\G_a \rtimes_\sigma \G_m$ on $\G_a \simeq \A^1$ yields an action on $P \simeq S_* \times \G_a$ extending to
$\hat P$. Its orbits have at most dimension 1, since that holds on $P$ and $P$ is dense in $\hat P$. 
Assume that there is a nontrivial $\G_a$-orbit $\G_a * x \hookrightarrow E$.
Then we have even $$(\G_a \rtimes_\sigma \G_m)x = \G_a * x,$$
the left hand side being irreducible and one-dimensional, hence it is the union of the singular point $p_0 \in E$ and a $\G_m$-orbit. But since $E$ is purely two dimensional, there are infinitely many such orbits -- a contradiction, since different orbits are disjoint.
\end{proof}
 
 \section{Pullbacks and Extensions of $\SL_2$}\label{sec:SL2}
Principal $\G_a$-bundles over $\A^2_*$, also studied in~\cite{DuFi11}, are classified by $H^1(\A^2_*,\cO_{\A^2_*})$. The ''most basic'' nontrivial among these is $\SL_2$, whose cocycle with respect to the open cover $\A^2_*=\A^2_x\cup\A^2_y$ is given by
$(xy)^{-1}\in H^1(\A^2_*,\cO_{\A^2_*})\simeq x^{-1}y^{-1}\C[x^{-1},y^{-1}].$
Proposition~\ref{pullback} states that every nontrivial principal $\G_a$-bundle $\pi\colon P\to S_*$ locally can be realized as a pullback of $\SL_2$ with respect to a certain morphism $\phi\colon U_*\to\A^2_*$, where $U$ is an affine open neighbourhood of $\pt$ and $U_*:=U\setminus\{\pt\}$. 
Using this representation of $P|_{U_*}\to U_*$ as a pullback of $\SL_2$ around $\pt\in S$, we obtain an affine extension $\hat\pi\colon \hat P\to S$ for every affine extension we can find for $\SL_2$; this is an direct consequence of Proposition~\ref{prop:pullbext}.

\begin{proposition}\label{pullback}
For any nontrivial principal $\G_a$-bundle $\pi\colon P \to S_*$ there is 
a neighbourhood $U$ of $\pt$ together with regular functions $g,h \in \cO (U)$ with $\pt$ as their only common zero, 
such that 
$$
P|_{U_*} \simeq \varphi^*(\SL_2)
$$ 
with the morphism $\varphi:= (g,h)\colon U_* \to \A^2_*$. 
\end{proposition}
\begin{proof}
We consider the subsheaf $\cF\subset \pi_*(\cO_P)$ on $S_*$ of the direct image sheaf which is defined for an open affine subset $V\subset S_*$ by
$$
\cF(V):=\{ f \in \cO (\pi^{-1}(V)); D^2(f)=0\}.
$$
Here $D$ denotes the locally nilpotent derivation which corresponds to the $\G_a$-action on $\pi^{-1}(V)\simeq V\times\G_a\subset P$ (see also Definition~\ref{melitta}). Since $\cF$ is locally free of rank 2, \cite[Cor. 4.1.1]{Hor64}
implies that there is a neighbourhood $U$ of $\pt$
such that
$$
\cF|_{U_*} \simeq (\cO_{U_*})^2.
$$
Denote by $f_0,f_1 \in \cF(U_*)$ the sections corresponding to 
$(1,0), (0,1) \in \cO (U_*)^2$.
 When restricted to a fiber over any point in $U_*$, the functions $f_0$ and $f_1$ are linearly independent polynomials of degree 1, since $D$ is the partial derivative with respect to the fiber variable, and it follows that $(0,0)$ cannot be contained in the image of  $(g,h):=(D(f_0),D(f_1))\colon U_* \to\A^2$. 
\par
Furthermore, from $D(gf_1-hf_0)=0$ we get that $gf_1-hf_0$ is a nowhere vanishing function $e$ on $\pi^{-1}(U_*)$ which is constant on the $\pi$-fibers, hence it is the pullback of a function $e \in \cO^* (U_*)$. Replacing $f_1$ with $e^{-1}f_1$, we may assume that $gf_1-hf_0=1$. It follows that 
\begin{eqnarray*}
P|_{U_*}&\stackrel\sim\longto&\varphi^*(\SL_2)=\{(w,(u,v))\in U_*\times\A^2\,\,;\,\, g(w)v-h(w)u=1\}\\
z&\mapsto& (\pi(z),f_0(z),f_1(z))
\end{eqnarray*}
is an isomorphism. Note that the functions $g,h  \in \cO (U_*)=\cO(U)$ satisfy 
 $g(\pt)=0=h(\pt)$; otherwise $P|_{U_*}$ would be trivial as well as $P$ itself.
\end{proof}

\begin{remark}\label{rem:global}
Any nontrivial principal $\G_a$-bundle $\pi\colon P\to\A^2_*$ is isomorphic to a pullback $\phi^*(\SL_2)$ with a morphism $\phi:=(g,h)\colon\A^2_*\to\A^2_*$. For a proof of this ''global'' statement for $\G_a$-bundles on $\A^2_*$, we proceed as in the proof of Proposition~\ref{pullback}, obtain that $\cF$ extends to a locally free sheaf $\hat \cF$ on the plane and use the famous result of Quillen-Suslin \cite{TLam06} which states that locally free $\cO_{\A^2}$-modules are free. In our situation This means $\hat \cF\simeq(\cO_{\A^2})^2$.
\end{remark}

\begin{corollary}\label{cor:nontrivialaffine}
A principal $\G_a$-bundle over $S_*$ is affine if and only if it is nontrivial.
\end{corollary}
\begin{proof} $S_* \times \G_a$ is not affine. If $P\to S_*$ is nontrivial and $U$ affine, $\pi^{-1}(U)$ is affine as well. Indeed, it follows from Proposition~\ref{pullback} that
$$
\pi^{-1}(U)\simeq \varphi^*(\SL_2) \simeq   U \times_{\A^2} \SL_2
$$
with respect to the morphisms $\SL_2 \to \A^2_* \hookrightarrow \A^2$ and 
$(g,h)\colon U \to \A^2$. Note that for the last isomorphism it is essential that $g(\pt)=0=h(\pt)$. Thus the composite morphism $P\stackrel\pi\to S_*\hookrightarrow S$ is affine, hence $P$ itself as well.
\end{proof}
Our next result states that extensions of $\SL_2$ induce extensions of $P|_{U_*}$. Using this we also get a global extensions of $P$ by gluing.

\begin{proposition}
\label{prop:pullbext} Let $\hat \varphi=(g,h)\colon U \to \A^2$ be a morphism and $\pt \in U$ be the only common zero of $g,h \in \cO(U)$ and $P=\varphi^*(\SL_2)$ with
$\varphi:=\hat \varphi |_{U_*}$. 
If $\hat R$ is an affine extension of $\SL_2$, then the normalization of the reduction 
$\hat P$ of the pull back $ \hat\varphi^* (\hat R):= U \times_{\A^2} \hat R$ is an affine extension of $P$.
\end{proposition}
 
 \begin{proof} An extension $\hat P$ of $P$ is defined by completing the pullback diagram for $P$ into a cartesian diagram as follows. Here $\psi:=\varphi^*(\pi)$ and $\hat\psi:=\hat\varphi^*(\hat\pi)$.\\
\begin{tabular}{rrr}\hspace{15ex}
\begin{minipage}{0.25\textwidth}
\[\xymatrix{
P\ar[r]\ar[d]_{\psi}& \SL_2\ar^\pi[d]\\
U_* \ar^{\varphi}[r] & \A^2_*}\]
\end{minipage}
&
\begin{minipage}{0.1\textwidth}
$\hookrightarrow$
\end{minipage}
&
\begin{minipage}{0.1\textwidth}
\[\xymatrix{
\hat P\ar[r]\ar[d]_{\hat\psi}& \hat R\ar[d]^{\hat\pi}\\
U \ar^{\hat\varphi}[r] & \A^2}\]
\end{minipage}
\end{tabular}\par
By Lemma~\ref{interior}, the image $\hat\varphi(U)$ contains $0\in\A^2$ as an interior point, and it follows that $P \subset \hat P$ is dense. 
\end{proof}

\begin{lemma}\label{interior}
Denote by $\hat\varphi:=(g,h)\colon U\to\A^2$ the extension of the above morphism from the proof of Proposition~\ref{pullback}. Then the image of $\hat\varphi$ contains $0 \in\A^2$ as an interior point.
\end{lemma}
\begin{proof}
If $0 \in\hat\varphi(U)$ is not an interior point, there is an irreducible curve $C\subset\A^2$ through the origin, such that $C\cap \hat \varphi(U)$ is finite. With the embedding $\A^2\hookrightarrow\p^2,\,(x,y)\mapsto[x:y:1]$, $\hat\varphi$ induces a rational map $\overline \varphi \colon \overline U\to\p^2$, where $\overline U$ is some smooth projective closure of $U$. By Noether-Castelnuovo's classical theorem, there is a blowup $\xi\colon X\to\overline U$ and a projective morphism $\eta\colon X\to\p^2$ such that the following diagram commutes. 
\[\xymatrix{
&X\ar_\xi[dl]\ar^\eta[dr]&\\
\overline U \ar@{-->}^{\hat\varphi_P}[rr]&&\p^2\\
U\ar[rr]^{\hat\varphi}\ar@{^{(}->}[u]&&\A^2\ar@{^{(}->}[u]}\]
Note that $\xi\colon X\to\overline U$ is a finite composition of blowups at points above  $\overline U \setminus U$, so in particular we may think of $U$ as a subset of $X$.
Then the inverse image $\eta^{-1}(\overline C)$ consists of $\pt \in U$ and finitely many further points in $U$ and a closed subvariety of $X \setminus U$. This is a contradiction, since $\eta^{-1}(\overline C)$ is the support of a divisor in $X$ and thus cannot have any isolated points.
\end{proof}

\begin{remark}
Unfortunately we don't know if all extensions of $P$ can be obtained in the above way, and second, if different extensions of $\SL_2$ induce different extensions of $P$. 
\end{remark}

\section{The graded algebra of an affine extension}\label{sec:gradedalgebra}
We now introduce the graded algebra, denoted $\gr_D(B)$ of an affine extension $\hat\pi\colon\hat P\to S$ of a principal $\G_a$-bundle $\pi\colon P\to S_*$. Motivated by Propositions~\ref{pullback}~and ~\ref{prop:pullbext} we will restrict our attention to affine extensions of $\SL_2$ in the remaining sections. Hence it would be enough to develop this algebraic tool for bundles over $\A^2_*$, but since it works completely analogously for the more general setting described in Definition~\ref{def:affext}, we formulate it in terms of a general punctured surface $S_*=S\setminus\{\pt\}$ instead.\par
We denote by $D$ the locally nilpotent derivation which corresponds to the structural $\G_a$-action on $P$.

\begin{definition}\label{melitta}

If $A$ is a $\C$-algebra with a locally nilpotent derivation $D\colon A \to A$, we define the $D$-filtration $(A_{\le \nu})_{\nu \in \N}$ of $A$ by $A_{\le \nu}:= \ker D^{\nu+1}$, and define the associated graded algebra $\gr_D(A)$ as
$$\gr_D (A):= \bigoplus_{\nu=0}^\infty A_{\le \nu}/  A_{\le \nu-1}.$$
The ''leading term'' $\gr(f) \in \gr_D(A)$ of $f\in A\setminus\{0\}$ is defined as
$$\gr (f):=f+A_{\leq\nu-1}\in\gr_D(A)_\nu,$$
where $\nu\in\N$ is the unique natural number such that $f\in\ker D^{\nu+1}\setminus\ker D^\nu$.

\end{definition}

\begin{remark} In our case, where $A$ is the algebra of a principal $\G_a$-bundle over $S_*$, the $A_{\leq 0}$-submodule  $A_{\le \nu} \subset A$ consists of the functions whose restriction to any fiber is a polynomial of degree $\le \nu$. In particular $A_{\leq 0}=\cO(S_*)\simeq\cO(S)$, so 
$$
\gr_D (A)\simeq \cO(S)\oplus\bigoplus_{\nu=1}^\infty A_{\le \nu}/  A_{\le \nu-1}.
$$
\end{remark}

We can always regard $\gr_D (A)$ as a subalgebra of the polynomial algebra $\cO(S)[t]$ in one indeterminate $t$ over $\cO(S)$ as follows.
\begin{proposition}
\label{prop:grad}
Let $D\colon A \to A$ be a locally nilpotent derivation of the $\C$-algebra $A$. Then the sequence of ideals $\fm_\nu$, or more precisely $\fm_\nu(A)$, defined by
$$\fm_\nu:= D^\nu (A_{\le \nu})  \hookrightarrow \cO(S)$$
is decreasing and satisfies $\fm_0=\cO(S)$, and $\fm_\nu \fm_\mu \subset \fm_{\nu+\mu}$. 
Furthermore we have
$$
\mathrm{ gr}_D(A) \simeq  \bigoplus_{\nu=0}^\infty \fm_\nu t^\nu \hookrightarrow \cO(S) [t]. 
$$
\end{proposition}
\begin{proof} The isomorphism is induced by
$$
\gr_D (A)_\nu \to \fm_\nu t^\nu,\quad a + A_{\le \nu-1} \mapsto \frac{D^\nu a}{\nu !} t^\nu.
$$
\end{proof}

\begin{example}\label{rem:ghideal}
Let $S=U$ and $A=\cO(P)$ as in Proposition \ref{pullback}, and let $f_0,f_1 \in A_{\le 1}$ denote functions whose restrictions generate the vector space of polynomials of degree $\le 1$ on any fiber. Then, taking $g=D(f_0)$ and $h=D(f_1)$ we have
$$
A_{\leq\nu}=\bigoplus_{\alpha\in\N^2,\,|\alpha|=\nu}\cO(S)f^\alpha
$$
with the notation $f^\alpha = f_0^{\alpha_0}f_1^{\alpha_1}$ for $\alpha=(\alpha_0,\alpha_1)\in\N^2$ and $|\alpha|=\alpha_0+\alpha_1$. Since  $D^\nu(f^\alpha)=\alpha ! g^{\alpha_0}h^{\alpha_1}$, we obtain
$$
\mathrm{gr}_D (A)=  \bigoplus_{\nu=0}^\infty \langle g,h\rangle^\nu t^\nu \subset \cO(S)[t].
$$
\end{example}

Let us now consider an affine extension $\hat\pi\colon\hat P=\Spec(B)\to S$ of a $\G_a$-principal bundle $P\to S_*$ with $\cO(P)=A$. Since $D(B)\subset B$, we can form $\gr_D(B)$, and we get an inclusion 
$$
\gr_D(B)= \cO(S) \oplus \bigoplus_{\nu=1}^\infty \fb_\nu t^\nu \subset \cO(S) \oplus \bigoplus_{\nu=1}^\infty \fm_\nu t^\nu=\gr_D(A),
$$ 
where $\fb_\nu = \fm_\nu(B)$.\par
\begin{lemma}\label{primary}
The ideals $\fm_\nu(B), \nu >0$, of an affine extension $\hat P=\Spec(B)$ are $\fm_\pt$-primary, i.e. they are supported in $\{\pt\}\subset S$, or $\gr_D(B)=\cO (S)[t]$.
\end{lemma}
\begin{proof}
First we note that if $B=\cO(S\times\G_a)$ is the algebra of a trivial $\G_a$-bundle over a variety $X$, we have $\fm_\nu(B)=B^{\G_a}=\cO(X)$ for all $\nu$, i.e. the ideals $\fm_\nu(B)$ have empty support. It follows from the definition that the ideal sequence $\fm_\nu(B_f)$ of a localization at an element $f\in B^{\G_a}$ satisfies $\fm_\nu(B_f)=\fm_\nu(B)B_f$. Now, in our situation, $B_f=A_f$ for all $f\in\cO(S)$ with $f(\pt)=0$, and $A_f$ is indeed the algebra of a trivial $\G_a$-bundle, since $S_*\setminus V(f)$ is affine. It follows that the $\fm_\nu(B)$ can have no support outside $\pt\in S$.
\end{proof}

\begin{remark}
If $B=\cO(S)[f_1,\dots,f_r]\subset A$, it follows that
$$\cO(S)[\gr(f_1),\dots,\gr(f_r)]\subset \gr_D(B),$$
but it is not clear that the last inclusion always is an equality; actually, it is not even clear that $\gr_D(B)$ has to be finitely generated.
\end{remark}

\begin{proposition}\label{Prop:gr-uniqueness}
The algebra $B=\cO (\hat P) \subset A$ of regular functions of an affine extension $\hat P$ is uniquely determined by $\gr_D(B)\subset\gr_D(A)$ if  $\gr_D(B)$ is generated by $\gr_D(B)_1$ as a $\cO(S)$-algebra.
\end{proposition}
\begin{proof}
If $\gr_D(B)$ is finitely generated, we can take the generators to be homogeneous, i.e. 
$$\gr_D(B)=\cO(S)[\gr(f_1),\dots,\gr(f_r)]$$ 
for some $f_1,\dots,f_r\in B_{\le 1}$. Then it follows that $B=\cO(S)[f_1,\dots,f_r]$. Indeed $B_{\le n} \subset \cO(S)[f_1,\dots,f_r]$ holds by induction for all $n \in \N$. Let $n=1$ and $f \in A_{\le 1}$. Then $f \in B \Leftrightarrow \gr (f) \in \gr_D(B)_1$ since $B_{\le 0}=A_{\le 0}= \cO(S)$. This settles the case $n=1$, and the induction step follows from the assumption on $\gr_D(B)$.
\end{proof}

Another consequence if the graded algebra of an affine extension $\hat P=\Spec(B)$ of a principal $\G_a$-bundle $\pi\colon P\to S_*$ is generated in degree 1, concerns the $\G_a$-action on the exceptional fiber $E\subset\hat P$.

 \begin{proposition}
\label{extrivial} 
If $\hat P \not\simeq S \times \G_a$ and $\gr_D (B)$ is generated in degree 1, then the exceptional fiber $E \hookrightarrow \hat P$ consists of fixed points only.
\end{proposition}

\begin{proof} Choose generators $\gr (f_1), \dots, \gr (f_r) \in \gr_D(B)_1$. Then
$$\psi\colon \hat P  \to S \times \A^r, z  \mapsto (\pi (z),f_1(z),\dots,f_r(z))$$
is an equivariant embedding, when we endow the right hand side with the $\G_a$-action
\begin{eqnarray*}
(S\times\A^r)\times\G_a&\longto&S\times\A^r\\
(y,u_1,\dots,u_r,\tau)&\mapsto&(y,u_1+\tau g_1(y),\dots, u_r + \tau g_r (y)),
\end{eqnarray*}
where $g_i:=Df_i$ regarded as function on $S$. Assume that $g_i (\pt) \not= 0$
for some $i$. Then $\hat P \to S$ admits a trivialization over some  neighbourhood of $\pt\in S$. Hence it is a principal $\G_a$-bundle and thus $\hat P \simeq S \times \G_a$ since $S$ is affine.

We remark that $\psi (\pi^{-1}(y))\subset\{y\}\times\A^r$ is an affine line for $y \in S_*$ and that the exceptional fiber consists of all lines in $\{\pt\} \times \A^r$, which are limits of such lines.  
\end{proof}

\begin{remark}
For $B \subset \tilde B\subset A$ we have $B=\tilde B \Leftrightarrow\gr_D(B)=\gr_D(\tilde B)$. Namely, if $\tilde b\in\tilde B_{\leq\nu}$, there exists $b\in B_{\leq\nu}$ such that $\gr(\tilde b)=\gr(b)$, and hence $b-\tilde b\in \tilde B_{\leq\nu-1}=B_{\leq\nu-1}$. It follows by induction that $\tilde b \in B_{\le \nu}$.
\end{remark}

From now on, we specialize to $P=\SL_2 \to \A^2_* \subset \A^2$. Writing a matrix in $\SL_2$ as $\left( \begin{array}{cc} x & u \\ y & v \\ \end{array} \right)$ we may take
$f_0=u, f_1=v$, whence $g=x, h=y$. Thus $A = \cO (\SL_2)$ satisfies
$$
\gr_D(A)= \bigoplus_{\nu=0}^\infty \langle x,y\rangle^\nu t^\nu \subset \C [x,y][t].
$$
From Lemma~\ref{primary} we know that the graded subalgebra $\gr_D(B)$ of an affine extension $\hat \pi\colon \hat P\to\A^2$ is of the form
$$
C=\C[x,y]\oplus\bigoplus_{\nu=1}^\infty\mathfrak c_\nu t^\nu\subset  \C [x,y][t]
$$
with ideals $\mathfrak c_\nu\subset\langle x,y\rangle^\nu$ for all $\nu$.

\begin{question}\label{q:graded} For which decreasing sequences $(\fc_\nu)_{\nu\in\N_{>0}}$ of $\langle x,y\rangle$-primary ideals in $\C[x,y]$, is $C=\C[x,y]\oplus\bigoplus_{\nu=1}^\infty \fc_\nu t^\nu\subset\gr_D(A)$ equal to $\gr_D(B)$ for some affine extension $\hat P=\Spec(B)$ of the principal $\G_a$-bundle $\pi\colon \SL_2 \to\A^2_*$?
\end{question}

Question~\ref{q:graded} is partially answered by Theorem~\ref{Th: MT}, whose proof is the topic of sections~\ref{sec:family1} (part $1$) and~\ref{sec:family2} (part $2$). It gives the answer for two families of graded subalgebras of $\gr_D(\cO(\SL_2))$, which actually occur as the graded algebras of certain extensions of $\pi\colon\SL_2\to\A^2_*$.

\begin{theorem}
\label{Th: MT} 
Let $p,q \in \N_{>0}, \gcd(p,q)=1$ and $n \in \N_{>0}$.\\
$(1)$ There is a uniquely defined affine extension $\hat P_n$ of $\SL_2$ with ideals given by 
$$
\mathfrak c_\nu=\fm_\nu(\hat P_n)=\langle x,y\rangle^{(n+2)\nu}.
$$
$(2)$ There is an affine extension $\hat P(p,q)$ of $\SL_2$ which satisfies 
$$
\mathfrak c_\nu=\fm_\nu(\hat P(p,q))=\bigoplus_{p\alpha+q\beta\geq(p+q)\nu}\C x^\alpha y^\beta.$$
\end{theorem}

\begin{remark}\label{Rem:MT}
\begin{enumerate}[$(1)$]
\item Note that we may extend the first family by taking $\hat P_0:=\hat P(1,1)$.
\item The graded subalgebra $\C[x,y]\oplus\bigoplus_{\nu=1}^\infty\langle x,y\rangle^{(n+2)\nu} t^\nu\subset\gr_D(\cO(\SL_2))$ given by the ideal sequence $\fm_\nu(\hat P_n)$ is generated in degree 1, and hence uniquely determines the extension $\hat P_n$ by Proposition~\ref{Prop:gr-uniqueness}.
\item In the second family, we see that $x^3\in \fm_2(\hat P(2,1))$ but $x^3\notin \fm_1(\hat P(2,1))^2$. It follows that the algebra given by the above ideal sequence $\fm_\nu(\hat P(2,1))$ is not generated by its elements in degree 1.
\end{enumerate}
\end{remark}

\section{The first family of $\SL_2$-extensions}\label{sec:family1}
We prove part $(1)$ of Theorem~\ref{Th: MT} by constructing a family $\hat P_n$, indexed by $n>0$, of affine extensions of $\SL_2$ with $\fm_\nu(\hat P_n)=\langle x,y\rangle^{(n+2)\nu}$ and in the end of the section we give the proof of Theorem~\ref{equimod}. In order to simplify notation, we fix the positive integer $n\in\N_{>0}$, and tacitly understand that most of the constructions in this section depend on $n$. For instance we write $\hat P$ rather than $\hat P_n$ although $\hat P$ does depend on $n$.

Let $B_2=\G_a\rtimes\G_m$ denote the semidirect product with group multiplication $(a,b).(c,d)=(a+b^2c,bd)$. We also denote this product by $\rho_{(c,d)}(a,b)=\lambda_{(a,b)}(c,d)$ (i.e. $\rho$ for right, $\lambda$ for left). Note that $B_2\simeq \A^1\times\A^1_*$ as a variety, and that it can be realized as a (Borel) subgroup of $\SL_2$ via
$$
B_2\hookrightarrow\SL_2,\quad
(\alpha,\beta)\mapsto\begin{pmatrix}
\beta&\beta^{-1}\alpha\\
0&\beta^{-1}\end{pmatrix}.
$$
Let $U_0=\{[x:y]\in\p^1,\,x\neq 0\}\simeq\A^1,\textrm{ and }U_1=\{[x:y]\in\p^1,\,y\neq 0\}\simeq\A^1$. As usual, we take $x/y$ and $x/y$ as coordinates on $U_0$ and $U_1$ respectively.
\begin{remark}\label{rem:p1bundle}
The map $\SL_2\to\p^1,\quad A=(a_{ij})\mapsto[a_{11}:a_{21}]$ realizes $\SL_2$ as a $B_2$-principal bundle over $\p^1$, with $B_2$-equivariant trivializations given by \\
\begin{tabular}{l|r}
\begin{minipage}{0.50\textwidth}
\begin{eqnarray*}
\tau_0\colon U_0\times B_2&\stackrel\sim\longto&(\SL_2)_x\\
(z,(u,v))&\mapsto&\begin{pmatrix}
1&0\\z&1
\end{pmatrix}\cdot
\begin{pmatrix}
v&v^{-1}u\\0&v^{-1}
\end{pmatrix}
\end{eqnarray*}
\end{minipage}
&
\begin{minipage}{0.48\textwidth}
\begin{eqnarray*}
\tau_1\colon U_1\times B_2&\stackrel\sim\longto&(\SL_2)_y\\
(z,(u,v))&\mapsto&
\begin{pmatrix} z & -1\\1&0\end{pmatrix}\cdot
\begin{pmatrix}
v&v^{-1}u\\0&v^{-1}
\end{pmatrix}.
\end{eqnarray*}
\end{minipage}
\end{tabular}
and transition function given by $\tau_1^{-1}\tau_0\colon (z,(u,v))\mapsto(z^{-1},(z,z).(u,v))$.
\end{remark}

First we want to endow $\A^2$ with two commuting $B_2$-actions (depending on $n$), one from the left, the other from the right, i.e. $(b_1x)b_2=b_1(xb_2)$ for all $b_1,b_2\in B_2$. The next step will be to find simultaneous left- and right $B_2$-embeddings $B_2\hookrightarrow\A^2$ with respect to the two $B_2$-actions on $\A^2$. As a first step, we consider the automorphism
$$
\phi\colon \A^1\times\A^1_* \to \A^1\times\A^1_*,\quad (x,y)\mapsto(xy^n,y).
$$
Since $B_2= \A^1\times\A^1_*$ as a variety, we may conjugate the group multiplication with $\phi$ and obtain $B_2$-actions on $\A^1\times\A^1_*$ as follows.
\begin{definition}\label{def:LRactions}We define the $B_2$-actions $*_L$ and $*_R$ by\\
\begin{tabular}{l|r}
\begin{minipage}{0.50\textwidth}
\begin{eqnarray*}
*_L\colon B_2\times (\A^1\times\A^1_*)&\to&\A^1\times\A^1_*\\
((s,t),(x,y))&\mapsto& (s,t)*_L(x,y):=\\
&& (\phi\lambda_{(s,t)}\phi^{-1})(x,y)
\end{eqnarray*}
\end{minipage}
&
\begin{minipage}{0.48\textwidth}
\begin{eqnarray*}
*_R\colon(\A^1\times\A^1_*)\times B_2&\to& \A^1\times\A^1_*\\
((x,y),(s,t))&\mapsto&(x,y)*_R(s,t):=\\
&&(\phi\rho_{(s,t)}\phi^{-1})(x,y)
\end{eqnarray*}
\end{minipage}
\end{tabular}
\end{definition}
\begin{proposition} The actions $*_L$ and $*_R$ admit extensions
to a left action $\hat *_L$ and a right action $\hat *_R$ on $\A^2\supset \A^1\times\A^1_*$.
\end{proposition}

\begin{proof}
It follows from Definition~\ref{def:LRactions} that $*_L$ and $*_R$ are given as follows for $(s,t)\in B_2$ and $(x,y)\in\A^1\times\A^1_*$\\
\begin{tabular}{l|r}
\begin{minipage}{0.50\textwidth}
\begin{eqnarray*}
(s,t) *_L(x,y)&=&(st^ny^n+t^{n+2}x,ty)
\end{eqnarray*}
\end{minipage}
&
\begin{minipage}{0.48\textwidth}
\begin{eqnarray*}
(x,y) *_R(s,t)&=&(t^n(x+y^{n+2}s),ty),
\end{eqnarray*}
\end{minipage}
\end{tabular}\vspace{1ex}
and these are obviously defined even for $y=0$.
\end{proof} We will use the notation $*_L$ and $*_R$ even for the extended actions.

\begin{remark} The morphism $\phi\colon B_2 \hookrightarrow \A^1 \times \A^1_* \subset \A^2$ realizes $\A^2$ 
both as a left- and a right $B_2$-embedding with respect to the $B_2$-actions on $\A^2$ given by $*_L$ and $*_R$. We shall treat that map as an inclusion and write $B_2 \subset \A^2$.
\end{remark}

Now we use this $B_2$-embedding in order to define a fiber bundle $Q\to\p^1$ associated to the fiber bundle in Remark~\ref{rem:p1bundle} as follows.
\begin{definition}\label{def:Q}
We define
$$
Q:=\SL_2 \times^{B_2} \A^2,
$$ where $\A^2$ is endowed with the left- and right $B_2$-actions $*_L$ and $*_R$.

\end{definition}
This means that as a set, $Q$ is the orbit space with respect to the action
$$
B_2 \times (\SL_2 \times \A^2) \to \SL_2 \times \A^2,\quad (b,(x,y)) \mapsto (xb^{-1},by),
$$
while it is obtained as a variety by taking the locally trivial fiber bundle from Remark~\ref{rem:p1bundle} and replacing the general fiber $B_2$ by $\A^2$.

\begin{proposition}\label{prop:sl2embedding}
The action of $\SL_2$ by left multiplication on itself induces an $\SL_2$-action on $Q$. The natural inclusion $\SL_2\subset Q$ coming from $B_2$ is thus an $\SL_2$-embedding. 
\end{proposition}

In order to prove Proposition~\ref{prop:sl2embedding}, and also in order to be able to compute $\cO(Q)$, we present the explicit description of $Q$ in terms of gluing data:\par
Let $U_i\subset\p^1$ be as above for $i=0,1$, let $Q_0=U_0\times\A^2,\, Q_1=U_1\times\A^2$, and finally let $V_i=(U_0\cap U_1)\times\A^2\subset U_i\times\A^2.$ Then $Q$ is the variety obtained by gluing $Q_0$ and $Q_1$ along $V_0$ and $V_1$ via the morphism
\begin{eqnarray*}
V_0&\to& V_1\\
(z,(u,v))&\mapsto&(z^{-1},(z,z)*_L(u,v))\\
&&=(z^{-1},(z^{n+1}v^n+z^{n+2}u,z v)).
\end{eqnarray*}
The inverse morphism is given by
\begin{eqnarray*}
V_1&\to& V_0\\
(z,(u,v))&\mapsto&(z^{-1},(-z,z)*_L(u,v))\\
&&=(z^{-1},(-z^{n+1}v^n+z^{n+2}u,zv)).
\end{eqnarray*}

\begin{proof}[Proof of Proposition~$\ref{prop:sl2embedding}$]
The claim follows from the following formulas, which show that the $\SL_2$-action is algebraic. The matrix  $A=\begin{pmatrix}a&b\\c&d\end{pmatrix}\in\SL_2$ acts on $(z,(u,v))\in Q_0$ as
$$A(z,(u,v)):=\begin{cases}\left(\frac{c+dz}{a+bz},(b(a+bz),a+bz)*_L(u,v)\right)\in Q_0&\textrm{ if }a+bz\neq 0\\
\left(\frac{a+bz}{c+dz},(d(c+dz),c+dz)*_L(u,v)\right)\in Q_1&\textrm{ if }c+dz\neq 0,
\end{cases}$$
and on $(z,(u,v))\in Q_1$ as
$$A(z,(u,v)):=\begin{cases}\left(\frac{cz+d}{az+b},(-a(az+b),az+b)*_L(u,v)\right)\in Q_0&\textrm{ if }az+b\neq 0\\
\left(\frac{az+b}{cz+d},(-c(cz+d),cz+d)*_L(u,v)\right)\in Q_1&\textrm{ if }cz+d\neq 0.
\end{cases}$$
\end{proof}

\begin{remark}\label{rem:actions}
The right $B_2$-action $*_R$ has the $Q_i$ as invariant subsets; it is given on $Q_i$ for $i=0,1$ by $$(z,(u,v))*_R(s,t)=(z,(u,v)*_R(s,t)),$$ and it is well defined because of the fact that the left- and the right action commute. The $\G_a$-action on $Q$ induced by $*_R$ via the inclusion $\G_a\simeq\G_a\times\{1\}\subset B_2$ is given on $Q_i$ by $(u,v)*_R(s,1)=(u+v^{n+2}s,v),$ and this action extends the  structural $\G_a$-bundle action on $\SL_2\subset Q$. 
\end{remark}

We now take the affinization of $Q$, i.e. we take $\Spec(\cO(Q))$. This construction is described in detail, and it turns out that it is given by the contraction of a curve $C\subset Q$ which is isomorphic to $\p^1$ to a point. Indeed the right $B_2$-action on $Q$
restricts to a $\G_m$-action (use the inclusion $\G_m=\{0\} \times \G_m \subset B_2$), which is fiber preserving and elliptic on every fiber $\A^2$. The curve $C$ then consists of the sources of that $\G_m$-action. 

\par Using the local chart description of $Q$ given before the proof of Proposition~\ref{prop:sl2embedding}, we see that each of the following $n+5$ functions $f_0,f_1,g_0,\dots,g_{n+1},h$ is a well defined regular function on $Q$. The first line gives their definitions on $Q_0$ and the second line gives their definitions on $Q_1$.\par

\begin{tabular}{l|l|r}\hspace {-5ex}
\begin{minipage}{0.26\textwidth}
\begin{eqnarray*}
f_i\colon Q&\longto&\C\\
(t,(u,v))&\mapsto& t^{1-i}v\\
(t,(u,v))&\mapsto& t^iv
\end{eqnarray*}

\end{minipage}
&
\begin{minipage}{0.41\textwidth}
\begin{eqnarray*}
g_i\colon Q&\longto&\C\\
(t,(u,v))&\mapsto& t^{n+2-i}u+t^{n+1-i}v^n\\
(t,(u,v))&\mapsto& t^iu.
\end{eqnarray*}
\end{minipage}
&
\begin{minipage}{0.39\textwidth}
\begin{eqnarray*}
h\colon Q&\longto&\C\\
(t,(u,v))&\mapsto& u\\
(t,(u,v))&\mapsto& t^{n+2}u-t^{n+1}v^n.
\end{eqnarray*}
\end{minipage}
\end{tabular}\vspace{2ex}
Inspired by algebraic relations between these functions, we define a variety as follows:
\begin{definition}\label{def: hatP}
Let $\hat P\hookrightarrow \A^{n+5}$ be the reduced affine variety which is given by the ideal
$$I:=\left\langle
\begin{array}{c}
(Y_iY_j-Y_kY_l)_{i+j=k+l},\\
(X_0Y_{i+1}-X_1Y_i)_{i=0,\dots,n},\\
(ZY_i+X_1^nY_{i+1}-Y_{i+1}Y_{n+1})_{i=0,\dots,n},\\
ZX_0+X_1^{n+1}-X_1Y_{n+1}
\end{array}
\right\rangle\subset\C[X_0,X_1,Y_0,\dots,Y_{n+1},Z].$$
\end{definition}
As a preparation for Proposition~\ref{prop:contraction}, where we study the affinization morphism of $Q$, we make the following observation.
\begin{lemma}\label{Lem:chart}
Let $S_0$ and $S_1$ be the closed subsets of $\hat P$ which are given respectively by $$S_0\colon\,\, X_1=Z=0\hspace{1.5cm}\textrm{ and }\hspace{1.5cm} S_1\colon\,\, X_0=Y_0=0.$$
Then $S_0\cap S_1=\{0\}$.
\end{lemma}

\begin{proof}
Suppose that $p=(a_0,a_1,b_0,\dots,b_{n+1},c)\in S_0\cap S_1$. Then $a_0=a_1=b_0=c=0$, and using the relations given by the ideal $I$, we get $b_{1}b_{n+1}=0$. If $b_{n+1}=0$, we get $b_i=0$ for all $i$ by induction since $b_i^2=b_{i-1}b_{i+1}$ for $i=n,n-1,\dots,1$. If $b_{n+1}\neq 0$, all $b_i$ are zero, since $b_{i+1}b_{n+1}=0$. We get $p=0\in\hat P$ in any case.
\end{proof}
\begin{remark} Using the relations given by $I$, one can check that $S_i\simeq\A^2$ for $i=0,1$. 
\end{remark}

\begin{proposition}\label{prop:contraction}
The morphism 
$$
\psi\colon Q \longto \hat P,\quad q \mapsto (f_0,f_1,g_0,\dots,g_{n+1},h)(q)
$$
 contracts the curve $C\simeq\p^1$ given by $u=v=0$ (in both $Q_0$ and $Q_1$) to $0\in\hat P$. The restriction $\psi|_{Q\setminus C}\colon Q\setminus C\stackrel\sim\longto\hat P\setminus\{0\}$ is an isomorphism.
\end{proposition}
\begin{proof}
We check that $\psi$ induces isomorphisms
$$
\hat P\setminus S_i \stackrel\sim\longto  Q_i\setminus(Q_i\cap C),\quad i=0,1.
$$
The restriction of $\psi$ to $Q_1$ is given by $$\psi|_{Q_1}\colon (t,(u,v))\mapsto(v,tv,u,tu,\dots,t^{n+1}u,t^{n+2}u-t^{n+1}v^n).$$ Note that $\psi(Q\setminus Q_1)$ is the image of $t=0$ in the chart $Q_0$. It follows that $\psi(Q\setminus Q_1)\subset S_1$ with $S_1$ as in Lemma~\ref{Lem:chart}, and we may define an inverse map locally:
\begin{eqnarray*}
\hat P\setminus S_1&\stackrel\sim\longto& Q_1\setminus(Q_1\cap C)\\
(a_0,a_1,b_0,\dots,b_{n+1},c)&\mapsto&\begin{cases}(a_1/a_0,b_0,a_0)\textrm{ if }a_0\neq 0\\
(b_1/b_0,b_0,a_0)\textrm{ if }b_0\neq 0.
\end{cases}
\end{eqnarray*}
Analogously, we define an inverse of $$\psi|_{Q_0}\colon(t,(u,v))\mapsto(tv,v,t^{n+2}u+t^{n+1}v^n,t^{n+1}u+t^nv^n,\dots,tu+v^n,u)$$ as follows:
\begin{eqnarray*}
\hat P\setminus S_0&\stackrel\sim\longto& Q_0\setminus(Q_0\cap C)\\
(a_0,a_1,b_0,\dots,b_{n+1},c)&\mapsto&\begin{cases}(a_0/a_1,c,a_1)\textrm{ if }a_1\neq 0\\
((b_{n+1}-a_1^n)/c,c,a_1)\textrm{ if }c\neq 0.
\end{cases}
\end{eqnarray*}
One can check, using the relations given by the ideal $I$ in Definition~\ref{def: hatP}, that these morphisms indeed are isomorphisms.
\end{proof}

\begin{remark}
The variety $\hat P$ is an $\SL_2$-embedding since $Q$ is, and we only contracted a curve of fixed-points in $Q\setminus\SL_2$. It is also clear that $\hat P$ is three dimensional with the origin as its only singular point.
\end{remark}

\begin{remark}\label{rem:pullback}
Restricting $\psi$ to $\SL_2$, we get
\begin{eqnarray*}
\psi|_{\SL_2}\colon\SL_2&\longto&\hat P\\
\begin{pmatrix} x & u\\ y & v\end{pmatrix}&\mapsto& (y,x,vy^{n+1},vxy^n,\dots,vx^{n+1},ux^{n+1}),
\end{eqnarray*} so in particular we have $\psi^*(f_0)=y$ and $\psi^*(f_1)=x$.
\end{remark}

We start our preparations for the proof of Proposition~\ref{funktionsiso}.
\begin{definition}
Let us define the bidegree of nonzero monomials in $\C[v,t,u]$ as $\bideg(ct^iv^ju^k)=(k,j)\in\N^2$,  $\bideg(0)=(-\infty,-\infty)$, and then we extend this to a function $\bideg\colon\C[v,t,u]\longto\N^2\cup\{(-\infty,-\infty)\}$ by taking the maximal bidegree of the terms, with respect to the lexicographical order. For example $\bideg(t^7u^5+2v^3u^4+3vu^5)=(5,1)$.
\end{definition}

\begin{lemma}\label{lem:reduction}
Let $F\in\C[t,v,u]$ be a nonzero polynomial with $\bideg$-leading term $ct^iv^ju^k$, where $i\leq j+kn+2k$. Then $F=\tilde F+L$ for some $$\tilde F\in\C[v,vt,u,ut,\dots,ut^{n+1},ut^{n+2}-v^nt^{n+1}]\quad\textrm{and}\quad L\in\C[t,v,u]$$ with $\bideg(L)<\bideg(F)$.
\end{lemma}
\begin{proof}
If $i\leq j$, we take $\tilde F=c(tv)^iv^{j-i}u^k$. If $i>j$, we find integers $q,r$ so that $i-j=(n+2)q+r$ with $0\leq r\leq n+1$ and $0\leq q\leq k-1$, and then we take $\tilde F=c(tv)^jut^r(ut^{n+2}-v^nt^{n+1})^qu^{k-q-1}$. In both cases $\bideg(F-\tilde F)<\bideg(F)$.
\end{proof}

\begin{proposition}
\label{funktionsiso}
The $n+5$ functions $f_0,f_1,g_0,\dots,g_{n+1},h$ generate $\cO(Q)$ as a $\C$-algebra, and $\psi^*\colon\cO(\hat P)\stackrel\sim\longto\cO(Q)$ is an isomorphism.
\end{proposition}

\begin{proof} Since $Q$ is normal and $C$ of codimension 2 we get a morphism
$$\cO(\hat P)\stackrel{\psi^*}\longto\cO(Q)\simeq\cO(Q\setminus C)\simeq\cO(\hat P\setminus\{0\}).$$ It is clearly injective, and the surjectivity follows from  $\cO(Q)=\C[f_0,f_1,g_0,\dots,g_{n+1},h]$, a fact that we now prove:\par

The regular functions on $Q$ are the elements of $\C(Q)$ which are defined everywhere on $Q_0$ and $Q_1$, i.e. they can be seen as the polynomial functions on $\cO(Q_1)=\C[u,v,t]$ which remain polynomial as functions on $Q_0$ after the coordinate change induced by the transition function $V_0\to V_1$ (c.f. Definition~\ref{def:Q}).\par

Let $F\in\cO(V_1)=\C[u,v,t]\subset\C(Q)$ be a nonzero regular function on $Q$ with $\bideg$-leading term $ct^iv^ju^k$. After the coordinate change, the $\bideg$-leading term of $F$ becomes
$$ct^{-i}(tv)^j(ut^{n+2})^k.$$ It follows in particular, since $F$ is regular on $Q$, that $t^{-i+j+kn+2k}$ is a polynomial, so $i\leq j+nk+2k$. Now we use 

Lemma~\ref{lem:reduction} in order to write $F=\tilde F+L\in\cO(V_1)$ with $\tilde F\in\C[v,vt,u,ut,\dots,ut^{n+1},ut^{n+2}-v^nt^{n+1}]$ and $\bideg(L)<\bideg(F)$. Repeating this procedure a finite number of times, we finally arrive at $L=0$ and the claim of the proposition follows.
\end{proof}

Finally, we announce the result which settles part $(1)$ of Theorem~\ref{Th: MT}

\begin{proposition}
We have $$\gr_D(\cO(\hat P))\simeq\C[x,y]\oplus\bigoplus_{\nu\geq 1}\langle x,y\rangle^{(n+2)\nu}t^\nu.$$
\end{proposition}
\begin{proof}
Using Remark~\ref{rem:actions}, we see that the $\G_a$-action on $\hat P$ corresponds to the derivation $D\colon\cO(\hat P)\longto\cO(\hat P)$ given by $D(f_i)=0,\,D(g_i)=f_1^if_0^{n+2-i},\,D(h)=f_1^{n+2}$. It follows that $D^\nu(\cO(\hat P)_{\leq \nu})=\langle f_0,f_1\rangle^{(n+2)\nu}\subset\cO(\hat P)$ for each $\nu\geq 0$. Hence
$$\gr_D(\cO(\hat P))=\C[f_0,f_1]\oplus\bigoplus_{\nu=1}^\infty\langle
f_0,f_1\rangle^{(n+2)\nu}t^\nu.$$ This finishes the proof since $\psi^*(f_0)=y$ and $\psi^*(f_1)=x$ (as in Remark~\ref{rem:pullback}).
\end{proof}

\begin{proposition}
The variety $\hat P$ is normal.
\end{proposition}
\begin{proof}
Suppose that $f\in\C(Q)$ is integral over $\cO(Q)$. Then it is in particular integral over $\cO_{Q,q}$ for each $q\in Q$, and since $Q$ is normal, it follows that $f\in\cap_{q\in Q}\cO_{Q,q}=\cO(Q)$. Thus $\cO(Q)\simeq\cO(\hat P)$ is integrally closed and since $\hat P$ is an affine variety, we are finished.
\end{proof}

\begin{proof}[Proof of Theorem~$\ref{equimod}$]
Let $B_n:=\cO (\hat P_n)\subset\cO(\SL_2)$ and suppose that the affine $\SL_2$-extension is given by $\hat R=\Spec (C)$. By Lemma~\ref{primary}, we can choose an $n$ such that $\langle x,y\rangle^{n+2} \subset \fm_1(\hat R)$. Then we have $(B_n)_{\le 1} \subset C_{\le 1}$ and thus $B_n \subset C$, since $B_n$ is, as a $\C [x,y]$-algebra, generated by $(B_n)_{\le 1}$.
\end{proof}

\section{The second family of $\SL_2$-extensions}\label{sec:family2}
In this section we prove part $(2)$ of Theorem~\ref{Th: MT} by constructing a family $\hat P(p,q)$ of affine extensions of $\SL_2$, depending on two relatively prime natural numbers $p,q\in\N_{> 0}$, such that
$$\fm_\nu(\hat P(p,q))=\bigoplus_{p\alpha+q\beta\geq(p+q)\nu}\C x^\alpha y^\beta.$$
The numbers $p$ and $q$ are fixed throughout this entire section, and we will write $\hat P$ rather than $\hat P(p,q)$ and similarly for other objects that are introduced.\par
The key observation in order to construct the $\SL_2$-extension $\hat P$, is that we can equip $\SL_2$ with a $\G_m$-action and then realize $\SL_2$ as a $\G_m$-fibration over the quotient variety. The $\G_m$-action that we will use is defined as follows for $\lambda\in\G_m$.

$$\begin{pmatrix}x&u\\y&v\end{pmatrix}\lambda:=\begin{pmatrix}\lambda^p x & \lambda^{-q}u \\ \lambda^q y & \lambda^{-p}v\end{pmatrix}.$$

\begin{proposition}\label{Prop:Gmquotient}
The quotient morphism with respect to this $\G_m$-action is given by
\begin{eqnarray*}
\psi\colon\SL_2&\to& Y\simeq\{(a,b,c)\in\A^3\,\,|\,\,ac=b^q(b-1)^p\}, \\
\begin{pmatrix}x&u\\y&v\end{pmatrix}&\mapsto&(x^qu^p,xv,y^pv^q),
\end{eqnarray*}
It is a $\G_m$-fibration which is a principal $\G_m$-bundle above the regular part of $Y$.
Indeed $Y$ has at most two singular points and it is smooth 
if and only if $p=q=1$.
\end{proposition}

\begin{proof}
Since $\gcd(p,q)=1,$ we find $m,n\in\Z$ with $mq-np=1$, and we find $\G_m$-equivariant trivializations given respectively by
\begin{eqnarray*}
Y_a\times\G_m&\longto&(\SL_2)_{xu}\\
((a,b,c),\lambda)&\mapsto&
\begin{pmatrix}a^m\lambda^p& a^{-n}\lambda^{-q}\\ 
(b-1)a^n\lambda^q&ba^{-m}\lambda^{-p}
\end{pmatrix}
\end{eqnarray*}
and
\begin{eqnarray*}
Y_c\times\G_m&\longto&(\SL_2)_{yv}\\
((a,b,c),\lambda)&\mapsto&
\begin{pmatrix}bc^{-m}\lambda^{p} & (b-1)c^n\lambda^{-q}\\c^{-n}\lambda^{q}&c^m\lambda^{-p}\end{pmatrix}.
\end{eqnarray*}
The variety $Y$ is smooth at all points except possibly $(0,0,0)$, which is singular if and only if $q>1$, and $(0,1,0)$, which is singular if and only if $p>1$.

\end{proof}

\begin{remark}
The transition function from the first chart to the second in the above proof is given by $(a,b,c,\lambda)\mapsto(a,b,c,(b-1)^mb^n\lambda),$ i.e. multiplication by $(b-1)^mb^n\in\G_m$.
\end{remark}

\begin{remark}
The $\G_m$-action on $\SL_2$ has possibly nontrivial stabilizers only along  
two orbits, namely $\psi^{-1}(0,1,0)$ and $\psi^{-1}(0,0,0)$ with stabilizers the groups $C_p$ and $C_q$ of $p$-th and $q$-th roots of unity respectively.
\end{remark}

\begin{definition} We define
 $$
 \hat P:= \SL_2  \times^{\G_m} \A^1
$$ 
\end{definition}
This definition is analogous to Definition~\ref{def:Q}, in the sense that $\hat P$ is the orbit space with respect to the action
$$
\G_m\times(\SL_2\times\A^1)\longto \SL_2\times\A^1,\quad (\lambda,(x,y))\mapsto(x\lambda^{-1},\lambda y).
$$
A main difference is that $\hat P=(\SL_2  \times \A^1)\quot \G_m$ is already an affine variety in this situation, $\G_m$ being reductive. We also get normality of $\hat P$ for free, since $\SL_2\times\A^1$ is normal.

\begin{remark}\label{rem:cOB} Intuitively $\hat P$ is again obtained from $\SL_2$ by replacing the fiber $\G_m$ in the fibration in Proposition~\ref{Prop:Gmquotient} by $\A^1\supset\G_m$ -- though the replacement process itself is not as obvious as in the case of a principal bundle. In any case $\cO(\hat P)\subset\cO(\SL_2)$ consists of those functions which are defined as $\lambda\in\G_m$ tends to 0.
\end{remark}
It is a straightforward verification to check that $$(A*s)\lambda=(A\lambda)*(\lambda^{-(p+q)}s)$$ holds, for the standard right $\G_a$-action on $\SL_2$ which is given by
$$\SL_2\times\G_a\to\SL_2,\quad (z,s)\mapsto z*s:=z\begin{pmatrix}1&s\\0&1\end{pmatrix}.$$ This is well known to be equivalent with the fact that the locally nilpotent derivation $D$ on $\cO(\SL_2)$ which corresponds to the $\G_a$-action is homogeneous of degree $p+q$ with respect to the grading which corresponds to the $\G_m$-action. This will be used in the proof of the following proposition which will settle part $(2)$ of Theorem~\ref{Th: MT}.
\begin{proposition} The $\G_a$-action on $\SL_2$ extends to $\hat P$ and for
$B:=\cO(\hat P)$ we have
$$
\fm_\nu(\hat P)=\bigoplus_{p\alpha+q\beta\geq(p+q)\nu}\C x^\alpha y^\beta.
$$
\end{proposition}

\begin{proof} The $\G_m$-action on $\SL_2$ corresponds to a $\Z$-grading
$
\cO(\SL_2)= \bigoplus_{\mu=-\infty}^\infty \cO(\SL_2)(\mu),
$
with respect to which the locally nilpotent derivation $D\colon \cO(\SL_2) \to \cO(\SL_2)$  is homogeneous of degree $p+q$. Since $x,y,u,v$ are of degree $p,q,-q,-p$ respectively, we have
$$
\cO(\SL_2)(\mu)=\sum_{p(i-\ell)+q(j-k)=\mu} \C x^iy^ju^k v^\ell
$$
and as in Example~\ref{rem:ghideal} we have 
$$\cO(\SL_2)_{\le \nu}= \bigoplus_{k+\ell= \nu} \C [x,y] u^kv^\ell$$
for the $D$-filtration of $\cO(\SL_2)$.\par

It follows from Remark~\ref{rem:cOB} and the definition of $\cO(\SL_2)(\mu)$ that $B=\bigoplus_{\mu=0}^\infty \cO(\SL_2)(\mu)$ is the non-negative part of the $\Z$-graded algebra $\cO(\SL_2)$. In particular, the $\G_a$-action on $\SL_2$ extends to an action on $\hat P$, since $B$ is $D$-invariant, $D$ being homogeneous of degree $p+q \ge 2$.\par
Let us now determine $\gr_D(B)$ using the $\G_m$-decomposition of $\cO(\SL_2)$. Since the locally nilpotent derivation $D\colon \cO(\SL_2) \to \cO(\SL_2)$ is homogeneous, the $\G_m$-grading descends to the associated graded algebra
$$\gr_D (\cO(\SL_2)) = \C [x,y] \oplus \bigoplus_{\nu=1}^\infty \langle x,y\rangle^\nu t^\nu,$$
where the $\G_m$-grading on $\C [x,y]$ satisfies $\deg (x)=p, \deg (y)=q$ and associates the degree $-(p+q)$ to the variable $t$ (though $t \not\in \gr_D(\cO(\SL_2))$). Indeed $xt= \gr (u)$ and $yt=\gr (v)$.

It follows that
$$
\gr_D(B)= \gr_D(\cO(\SL_2))_{\ge 0} \hookrightarrow \gr_D(\cO(\SL_2)),
$$ 
where the subscript refers to the $\G_m$-grading. In other words,
$$\gr_D (B)=\C [x,y] \oplus \bigoplus_{\nu=1}^\infty \fm_\nu t^\nu$$
with $$\fm_\nu=(\langle x,y\rangle^\nu)_{\ge \nu (p+q)} = \bigoplus_{\alpha p + \beta q \ge \nu (p+q)} \C  x^\alpha y^\beta.$$
\end{proof}

\section{Small fixed point sets}\label{sec:otherext}
It follows from Proposition~\ref{extrivial} that the exceptional fiber $E=\hat P_n\setminus\SL_2$ conists of $\G_a$-fixed points for the $\SL_2$-extensions in the family that was constructed in section~\ref{sec:family1}. For the extensions $\hat P=\Spec(B)$ from section~\ref{sec:family2} it goes the same. This follows from the fact that $D(f)\in \bigoplus_{\mu = p+q}^\infty \cO(\SL_2)(\mu)$ for all $f\in B$. But since $p+q>0$, this implies that the exceptional fiber $E=\hat P\setminus\SL_2$ consists of fixed points of the $\G_a$-action, as all functions of positive $\G_m$-degree vanish as $\lambda\in\G_m$ tends to 0.\par
In this section, we construct some extensions with empty fixed point set and with one dimensional fixed point set, taking the $\SL_2$-extension $\hat P(1,1)$ as starting point.
\begin{proposition}
\label{transaffin} Assume that the exceptional fiber $E \hookrightarrow Y= \Spec (B)$ 
of an affine extension $Y \to \A^2$ of $P \to \A^2_*$ is (the support of) a Cartier divisor and coincides with the fixed point set $E= Y^{\G_a}.$ Let $C \hookrightarrow E$ be a closed subvariety.
Denote $\Bl_C (Y) \to Y$ the blowup of $Y$ with center $C \hookrightarrow E$ and $\tilde E \hookrightarrow \Bl_C (Y)$ the strict transform of $E$. Then 
$$Y_1:= \Bl_C (Y) \setminus \tilde E$$
is an affine extension. 
\end{proposition}

\begin{proof} It is clear that $Y_1$ inherits a $\G_a$-action, since the center of the blowup is fixed by the $\G_a$-action on $Y$. It remains to show $Y_1$ is affine. It is enough to check that the morphism $Y_1\to Y$ is affine. After passing to a cover of affine open subsets of $\Spec(B)$, we may assume that $E \hookrightarrow Y$ is given by one function:
$I(E)= (f) \subset B=\cO (Y)$. But then, if $I(C)=(g_1,\dots,g_s) \ni f$ we have
$$Y_1= \Spec (B\big[\frac{g_1}{f}, \dots,\frac{g_s}{f}\big]),$$ see~\cite[Prop. 1.1]{KaZa99}.
\end{proof}

Now we start with $Y=\hat P_0 \supset \SL_2$ and take $Y_1:=\Bl_a(Y) \setminus \tilde E$ with the exceptional fiber $E \hookrightarrow \hat P_0$ and some point $a \in E$. Using the realization of $Y$ as locally trivial bundle over $\PP^1$, we see that
we may think of $a \in Y$ as the origin in
$$a=(0,0,0) \in  \A^3 = \Spec (\C[x,y,z])=:U$$
where $(x,y,z) \mapsto [1:z]$ is the bundle projection, while
$\A^1 \times \A^1_* \times \A^1$ is $\G_a \rtimes_\sigma \G_m \times \A^1$.
Furthermore the $\G_a$-action corresponds to
$$D\colon \C [x,y,z] \to \C [x,y,z], x \mapsto y^2, y \mapsto 0, z \mapsto 0$$ 
and $E \cap U=\A^1 \times 0 \times \A^1$. Then above $U$, in the blowup, we have
$$U_1:= \Spec(\C[x/y,y,z/y]),$$ with $D(x/y)=y,$ and $D(y)=0=D(z/y)$. If we take $\xi=x/y, \eta=y, \zeta=z/y$, we see that $\G_a$ acts linearly on $U_1=\A^3=
\Spec (\C[\xi, \eta, \zeta])$, namely $D= \eta \frac{\partial}{\partial \xi}.$

Note that
$$
Y_1 \supset U_1 \supset E_1= \A^1 \times 0 \times \A^1.
$$
Now let us apply the recipe of Proposition~\ref{transaffin} with some subvariety $C \hookrightarrow E_1$. We obtain
an affine extension $Y_2$. We discuss several choices of $C$:\par
$(1)$ If $C=\{ (0,0,0)\}$, the exceptional fiber is naturally isomorphic to
$$
E_2 \simeq \PP(\A^3) \setminus \PP (\A^1 \times 0 \times \A^1)
$$
with the restriction of the induced linear $\G_a$-action on
$\PP(\A^3) \simeq \PP(T_0(\A^3))$. Hence the $\G_a$-action on $Y_2$ is free.\par
$(2)$ Now let $C \hookrightarrow E_1= \A^1 \times 0 \times \A^1$
be a smooth curve. For the fiber $F_b$ over a point $b \in C$ there is a natural isomorphism 
$$
F_b \simeq \PP (\A^3/T_b(C)) \setminus \{ (\A^1 \times 0 \times \A^1)/T_b(C)\}.$$ 
We distinguish two cases:
\begin{enumerate}
\item If $T_b(C)=\C ( 1,0,0)$, the $\G_a$-action on $F_b$ is trivial.
\item If $T_b(C)=\C ( \alpha,0,\beta)$ with $\beta \not=0$, the $\G_a$-action on $F_b$ is free. So if $C$ is not a line parallel to $\C (1,0,0)$, the fixed point set has at most dimension one.
\end{enumerate}

\end{document}